\documentclass[twoside,a4paper,12pt]{article}
\usepackage{amssymb, amsmath, amsthm}
\usepackage{fancyhdr}
\usepackage{color}

\newdimen\myMargin
\textwidth \paperwidth
\advance\textwidth -2\myMargin
\textheight \paperheight
\advance\textheight -2\myMargin
\advance\oddsidemargin \myMargin
\advance\evensidemargin \myMargin
\advance\topmargin \myMargin
\advance\textheight -17 true mm
\advance\topmargin 12.5 true mm

\newcommand{\B}{ R}
\newcommand{\Bx}{\B\left[x_1,\dots,x_n\right]}
\newcommand{\Fp}{\mathbb F_p}

\newcommand{\F}{\mathbb F}
\newcommand{\stb}[2]{#1_1,\dots,#1_{#2}}

\newcommand{\N}{\mathbb N}

\newcommand{\ve}[1]{\mathbf{#1}}
\newcommand{\Fx}{\F\left[x_1,\dots,x_n\right]}
\newcommand{\Fxv}{\F\left[\ve x\right]}

\newcommand{\monom}[2]{\ve{#1}^{\ve{#2}}}

\newcommand{\x}{(\ve x)}

\newtheorem{theorem}{Theorem}
\newtheorem{lemma}[theorem]{Lemma}

\newtheorem{conjecture}[theorem]{Conjecture}

\theoremstyle{remark}
\theoremstyle{remark}

\title{Some extensions of Alon's Nullstellensatz}

\author{\normalsize
  \begin{minipage}{0.3\linewidth}
    \large
    G\'eza K\'os \\
    \footnotesize
    Computer and Automation Research 
Institute, Hungarian Acad. Sci; \\
    Dept. of Analysis, E\"otv\"os 
Lor\'and Univ., Budapest \\
    \texttt{kosgeza@sztaki.hu} \\
    \normalsize
  \end{minipage}
  \qquad
  \begin{minipage}{0.3\linewidth}
    \large
    Tam\'as M\'esz\'aros  \\
    \footnotesize
    Dept. of Mathematics, \\
Central European University\\
 \texttt{Meszaros\_Tamas@ceu\_budapest.edu}
    \normalsize
  \end{minipage}
  \qquad
  \begin{minipage}{0.3\linewidth}
    \large
    Lajos R\'onyai \\
    \footnotesize
    Computer and Automation Research 
     Institute, Hungarian Acad. Sci. \\
    Dept. of Algebra, Budapest 
    Univ. of Technology and Economics \\
    \texttt{lajos@ilab.sztaki.hu}
    \normalsize
  \end{minipage}
  \vspace{0.5cm}
}

\begin{document}

\thispagestyle{empty}

\maketitle

\footnotetext{
\noindent
{\em Mathematics Subject Classification (MSC2010):} 05-XX, 05E40, 12D10. 
Research supported in part by OTKA grants NK 72845, K77476, and K77778. 
}

%==========================================================================
{\em This paper is dedicated to professors
 K\'alm\'an Gy\H{o}ry, Attila Peth\H{o},
J\'anos Pintz, and Andr\'as S\'ark\"ozy, on the occassion of their (round)
birthdays.}

\begin{abstract}
Alon's combinatorial Nullstellensatz and
in particular the resulting nonvanishing criterion 
is one of the most powerful algebraic tools 
in combinatorics, with many important applications. 
The nonvanishing theorem has been extended in two directions. The first and
the third named authors proved a version allowing 
multiple points. Micha{\l}ek established a variant which is valid over 
arbitrary commutative rings, not merely over subrings of fields. In this
paper we
give new proofs of the latter two results and provide a common
generalization of them. 
As an application, we prove extensions 
of the theorem of Alon and F\"uredi on hyperplane coverings of discrete cubes.  
\end{abstract}

%%%%%%%%%%%%%%%%%%%%%%%%%%%%%%%%%%%%%%%%%%%%%%%%%%%%%%%%

\section{Introduction}

Alon's combinatorial Nullstellensatz (Theorem 1.1 from \cite{Alon1}) 
and the resulting nonvanishing criterion (Theorem 1.2 from \cite{Alon1}
is one of the most powerful algebraic tools in combinatorics. It 
has several 
beautiful and strong applications, see \cite{CLM}, \cite{Felszeghy},
\cite{GT}, \cite{KarolyiI}, \cite{KarolyiR}, \cite{PS}, \cite{Sun} for some 
recent examples. 

Let $\F$ be a field, $S_1,S_2, \ldots , S_n$ be finite nonempty subsets 
of $\F$. Let $\Fxv=\Fx$ stand for the ring of polynomials over $\F$ in
variables $x_1, \ldots, x_n$.  Alon's theorem is a specialized, precise 
version of the Hilbertsche Nullstellensatz for the ideal $I(S)$ 
 of all polynomial 
functions
vanishing on the set $S=S_1\times S_2\times \cdots \times S_n\subseteq
\F^n$, and for the basis  $f_1,f_2,\ldots ,f_n$ of $I(S)$, where 
$$ f_i=f_i(x_i)=\prod_{s\in S_i}(x_i-s)\in \Fxv $$ 
for $i=1,\ldots ,n$. From 
this  a simple and widely applicable nonvanishing 
criterion (Theorem 1.2 in \cite{Alon1}) has been deduced. It provides a 
sufficient condition for a
polynomial $f\in \Fxv$ for not vanishing everywhere on 
$S$. 

\medskip

Herewith we consider extensions of the latter result in two directions. 
Before formulating these, we need first some notation and definitions.
Let $\N$ denote the set of nonnegative integers, and let $n$ be a
fixed positive integer. Throughout the paper $\B$ will denote a commutative 
ring (with 1, as usual), and $\F$ stands for a field.
Vectors of length $n$ are denoted by boldface
letters, for example $\ve s=(\stb sn)\in\B^n$ stands for points in the
space $\B^n$. For vectors $\ve a,\ve b\in \N^n$, the relation $\ve
a\ge\ve b$ etc. means that the relation holds at every component. We
use the same notation for constant
vectors. e.g. $\ve0=(0,0,\ldots,0)$ or $\ve1=(1,1,\ldots,1)$.

For $\ve w\in\N^n$, we write $\monom xw$ for the monomial
$x_1^{w_1}\dots x_n^{w_n}\in\Bx$.  If $\ve s\in\B^n$, then $(\ve
x-\ve s)^{\ve w}$ stands for the polynomial
$(x_1-s_1)^{w_1}\dots(x_n-s_n)^{w_n}$.

\bigskip%-------------------------------------------------------------

It is well known that for an arbitrary $\ve s\in \B^n$ we can express
a polynomial $f(\ve x)\in \Bx$ as
\begin{equation} \label{kifejt} f(\ve x) = \sum_{\ve u \in \N^n} f_{\ve u}(\ve
  s) (\ve x - \ve s)^{\ve u},
\end{equation} 
where the coefficients $f_{\ve u}(\ve s)\in \B$ are uniquely
determined by $f$, $\ve u$ and~$\ve s$.  In particular we have
$f_{\ve0}(\ve s)=f(\ve s)$ for all $\ve s \in\B^n$. 
Observe that if $u_1+\dots+u_n\ge\deg f$, then $f_{\ve u}=f_{\ve
  u}(\ve s)$ does not depend on $\ve s$.

\bigskip%-------------------------------------------------------------

Suppose now that  $S_1, S_2, \ldots ,S_n$ are nonempty finite subsets 
of $\B$, and assume  further that we have 
a positive integer {\em multiplicity} $m_i(s)$ attached to every element 
$s\in S_i$. This way we can view the pair $(S_i,m_i)$ as a multiset which
contains the element  $s\in  S_i$ precisely $m_i(s)$ times. We shall
consider the sum $d_i=d(S_i):=\sum\limits_{s\in S_i}m_i(s)$ as the size 
of the multiset $(S_i,m_i)$.
We put 
$S=S_1\times S_2\times \cdots S_n$.
For an element $\ve s =(\stb sn) \in S$ we set the multiplicity vector 
$m(\ve s)$ as $(m_1(s_1),\ldots ,m_n(s_n) )$, and write 
$|m(\ve s)|=m_1(s_1)+\cdots +m_n(s_n)$. 

\medskip

We formulate first a version of Alon's powerful nonvanishing theorem 
(Theorem 1.2 in \cite{Alon1}) for multiple points over fields.  From this
one can obtain Alon's result by setting $m_i(s)=1$ identically.

\begin{theorem} \label{nonvanish}
Let $\F$ be a field, $f=f(x_1,\ldots ,x_n)\in \F [x_1,\ldots ,x_n]$ be a
polynomial of degree $\sum\limits_{i=1}^n t_i$, where each $t_i$ is a nonnegative
integer. Assume, that the coefficient in $f$ of the monomial 
$x_1^{t_1}x_2^{t_2}\cdots x_n^{t_n}$ is nonzero.  Suppose further that
$(S_1, m_1), (S_2,m_2), \ldots ,(S_n, m_n)$ are multisets of $\F$ such that 
for the size $d_i$ of $(S_i,m_i)$ we have $d_i>t_i$ ($i=1, \ldots , n$).
Then there exists a point $\ve s =(s_1, \ldots, s_n)\in S_1\times \cdots
\times S_n$ and an exponent vector $\ve u =(u_1, \ldots, u_n)$ with 
$u_i<m_i(s_i)$ for each $i$, such that $f_{\ve u}(\ve s) \not = 0$.
\end{theorem}

\medskip

Theorem \ref{nonvanish} was first proved in \cite{KR}.
In another direction, over an arbitrary  commutative ring $\B$ Micha{\l}ek 
proved the
following extension of the nonvanishing theorem:

\begin{theorem}\label{theorem 1}
Let $\B$ be a commutative ring, and let $f\x$ be a
polynomial in $\Bx$. Suppose the degree $deg(g)$ of $f$ is
$\sum_{i=1}^n t_i$, where $t_i$ is a nonnegative integer, and
suppose the coefficient of $\prod_{i=1}^n x_i^{t_i}$ in $f$ is
nonzero. Suppose further that  $S_1,\dots,S_n$ are subsets of $\B$ with
$|S_i|>t_i$, and with the property that if $s\not =s^* \in S_i$, then
$s-s^*$ is a unit in $\B$. Then there exists a vector  
$\ve s \in S=S_1\times \dots \times S_n$, such that

\[ f(\ve s) \neq 0.\]

\end{theorem}

In the case $\B=\F$, when we work over a field, the above result
specializes to Alon's nonvanishing theorem; note that if $\B$ is a field, 
then $s-s^*$ is always a unit, whenever $s$ and $s^*$ are different. 
We give two new proofs of Theorem \ref{theorem
1}. The first one will be on the basis of an identity obtained from an
interpolation argument. The second proof follows closely the original line
of reasoning by Alon.  
In fact, for a subset $S$ of $\B^n$ let us denote the set of polynomials
from $\Bx$ vanishing at all $s\in S$ by $I(S)$. It is easy to see that $I(S)$
is an ideal of $\Bx$.
Now Theorem \ref{theorem 1} will be a simple consequence 
of the following  result, which can be considered as an extension of Alon's 
Nullstellensatz.

\begin{theorem}\label{theorem 2}
Let $\B$ be a commutative ring, and let $f\x$ be a
polynomial from $\Bx$. Let $S_1,\dots,S_n$ be nonempty finite subsets of
$\B$ with the property that if $s\not= s^* \in S_i$, then $s-s^*$ is a unit
in $\B$, $S=S_1\times \dots \times S_n$ and define
$g_i(x_i)=\prod_{s\in S_i} (x_i-s)$. 
Then for every polynomial $f\x\in \Bx$ there are polynomials
$h_1,\dots, h_n,r \in \Bx$ such that $deg h_i \leq deg f- |S_i|$ 
for all $i$, the degree of $r$ is less then $|S_i|$ in every
$x_i$, for which
\[f\x =r\x + \sum_{i=1}^{n} h_i\x g_i(x_i).\]
Moreover, $f\in I(S)$ if and only if $r$ is identically zero, hence  
$g_1(x_1),\ldots ,g_n(x_n)$ is a basis of $I(S)$. 
\end{theorem}

\bigskip
We have the following common generalization of Theorems \ref{nonvanish}
and \ref{theorem 1}:

\begin{theorem}\label{modulo_multi}
Let $R$ be a ring, $f=f(x_1,\ldots ,x_n)\in R [x_1,\ldots ,x_n]$ be a
polynomial of degree $\sum\limits_{i=1}^n t_i$, where each $t_i$ is a
nonnegative integer. Assume, that the coefficient in $f$ of the monomial 
$x_1^{t_1}x_2^{t_2}\cdots x_n^{t_n}$ is nonzero.  Suppose further that
$(S_1, m_1), (S_2,m_2), \ldots ,(S_n, m_n)$ are multisets of $R$ such that
for the size $d_i$ of $(S_i,m_i)$ we have $d_i>t_i$ ($i=1, \ldots , n$), 
and for each $i$ any nonzero element $s-s^*$ from $S-S$ is a unit in $R$.
Then there exists a point $\ve s =(s_1, \ldots, s_n)\in S_1\times \cdots
\times S_n$ and an exponent vector $\ve u =(u_1, \ldots, u_n)$ with
$u_i<m_i(s_i)$ for each $i$, such that $f_{\ve u}(\ve s) \not = 0$.
\end{theorem}

\bigskip

In the next section we prove Theorem 
and~\ref{modulo_multi}, which implies Theorems 
\ref{nonvanish}, and \ref{theorem 1}.
The proof will be based on an argument extending  
univariate Hermite  interpolation to our setting.
In Section 3 we give an alternative proof for
Theorem~\ref{theorem 1}, wich will follow from  Theorem~\ref{theorem 2}.
This line of reasoning is an adaptation of Alon's original proofs 
from  \cite{Alon1}. In Section 4 we give two applications. These will be
extensions of a theorem of Alon and F\"uredi on almost covering a discrete 
hypercube by hyperplanes. In Section 5 some concluding remarks are given.

\section{Proof of Theorem \ref{modulo_multi} }

\begin{lemma}
  \label{lem:h_uniq}
  Let $(S,m)$ be a nonempty multiset in a commutative ring $R$ such
  that all nonzero elements $s-s^*$ in $S-S$ are units, and let
  $g(x)=\prod\limits_{s\in S}(x-s)^{m(s)}$.  For all polynomials
  $f(x)\in R[x]$, the following statements are equivalent:

  (a) $f_u(s)=0$ for every $s\in S$ and $0\le u<m(s)$;

  (b) the polynomial $(x-s)^{m(s)}$ divides $f(x)$ for every $s\in S$;

  (c) $g(x)$ divides $f(x)$.
\end{lemma}  

\begin{proof}
  The relations $(a)\Leftrightarrow(b)$ and $(c)\Rightarrow(b)$ are
  trivial. To prove $(a,b)\Rightarrow(c)$, apply an induction on
  $|S|$. The initial case $|S|=1$ is trivial.

  Let $n\ge2$ and assume the statement of the Lemma for $|S|=n-1$. Choose
  an $s_0\in S$ arbitrarily. By (b), there is a polynomial $f^*\in
  R[x]$ such that $f(x)=(x-s_0)^{m(s_0)}f^*(x)$. Let
  $g^*(x)=\prod\limits_{s\in S\setminus\{s_0\}}(x-s)^{m(s)}$.  We have
  to prove that $g^*$ divides $f^*$.

  First we show that $f^*_u(s)=0$ for every $s\in S\setminus\{s_0\}$
  and $0\le u<m(s)$. Suppose the contrary, and take a pair $(s,u)$ for
  which $f^*_u(s)\ne0$ and $u$ is minimal,
  i.e. $f^*_0(s)=f^*_1(s)=\ldots=f^*_{u-1}(s)=0$. Then
  $$
  f_u(s) = \Big( (x-s_0)^{m(s_0)} f^*(x) \Big)_u(s)
  = \sum_{v=0}^u \Big((x-s_0)^{m(s_0)}\Big)_v(s) \cdot f^*_{u-v}(s)
  = (s-s_0)^{m(s)} f^*_u(s).
  $$
  Since $s-s_0$ is a unit in $R$ and $f^*_u(s)\ne0$, this contradicts
  $f_u(s)=0$.

  So we have $f^*_u(s)=0$ for every $s\in S\setminus\{s_0\}$ and $0\le
  u<m(s)$. By the induction hypothesis, $f^*$ is divisible by $g^*$.
\end{proof}

%-------------------------------------------------------------------

\begin{lemma}
  \label{lem:h_base_pol}
  Let $(S,m)$ be a nonempty multiset in a commutative ring $R$ such
  that all nonzero elements $s-s^*$ in $S-S$ are units, and let
  $s_0\in S$ and $0\le u_0<m(s_0)$. Then there exists a polynomial
  $h^{(s_0,u_0)}(x)\in R[x]$ with $\deg h^{(s_0,u_0)}<d(S)$ such that
  for every $s\in S$ and $0\le u<m(s)$ we have
  $$
  h^{(s_0,u_0)}_u(s) = 
  \begin{cases}
    1 & \text{if $s=s_0$ and $u=u_0$;} \\
    0 & \text{otherwise.} \\
  \end{cases}
  $$
\end{lemma}

\begin{proof}
  For every $v=0,1,\dots,m(s_0)-1$ let
  $$
  f^{(v)}(x) =
  (x-s_0)^{v} \prod_{s\in S\setminus\{s_0\}}
  \left(\dfrac{x-s}{s_0-s}\right)^{m(s)}.
  $$
  These auxiliary polynomials have the following obvious properties:

  \begin{itemize}
  \item For every $s\in S\setminus\{s_0\}$ and every $v$, the
    polynomial $f^{(v)}(x)$ is divisible by $(x-s)^{m(s)}$, so
    $f^{(v)}_u(s)=0$ for all $0\le u<m(s)$.
  \item  For every $0\le u<v$, since $f^{(v)}(x)$ is divisible by
    $(x-s_0)^{v}$, we have $f^{(v)}_u(s_0)=0$.
  \item For every $v$ we have $f^{(v)}_v(s_0)=1$.
  \end{itemize}
  
  Now we can construct $h^{(s_0,u_0)}_u(s)$ as a linear combination of
  the auxiliary polynomials, inductively: if
  $h^{(s_0,u_0+1)},\ldots, h^{(s_0,m(s_0)-1)}$ are already defined
  then let
  $$
  h^{(s_0,u_0)}(x) = f^{(u_0)}(x) - \sum_{u_0<u<m(s_0)}
  f^{(u_0)}_u(s_0) \cdot h^{(s_0,u)}(x).
  $$
\end{proof}

%-------------------------------------------------------------------

\begin{lemma}[Hermite interpolation]
  \label{lem:h_int}
  Let $(S,m)$ be a multiset in a commutative ring $R$ such that all
  nonzero elements $s-s^*$ in $S-S$ are units. For each $s\in S$ and
  $0\le u<m(s)$, let $y_{s,u}$ be an arbitrary element in $R$. Then

  (a) there exists a unique polynomial $f(s)\in R[x]$, with $\deg
  f<d(S)$, satisfying $f_u(s)=y_{s,u}$ for every pair $(s,u)$;

  (b) this polynomial can be constructed as
  $$
  f = \sum_{s\in S} \sum_{u<m(s)} y_{s,u} h^{(s,u)}. 
  $$
\end{lemma}

\begin{proof}
  Let $f = \sum\limits_{s\in S} \sum\limits_{u<m(s)} y_{s,u}
  h^{(s,u)}$. For every $s\in S$ and $u<m(s)$ we have
  $$
  f_u(s)
  = \sum_{r\in S} \sum_{v<m(r)} y_{r,v} h^{(r,v)}_u(s)
  = \sum_{r\in S} \sum_{v<m(r)} \left\{ \begin{matrix}
      1 & \text{if $r=s$ and $v=u$} \\
      0 & \text{otherwise} \\
    \end{matrix} \right\} y_{r,v}
  = y_{s,u}
  $$
  so the polynomial $f$ satisfies the requested property.

  \medskip

  For the uniqueness, suppose that there exists another polynomial
  $f^*$ with the same property. Then, for all $s\in S$ and $u<m(s)$ we
  have $(f-f^*)_u(s) = f_u(s)-f^*_u(s) = 0$. By
  Lemma~\ref{lem:h_uniq}, this implies that $f-f^*$ is divisible by
  $\prod\limits_{s\in S}(x-s)^{m(s)}$. Since the degree of the latter
  polynomial is $d(S)$ and its leading coefficient is $1$, this
  contradicts $\deg(f-f^*)<d(S)$.
\end{proof}

%-------------------------------------------------------------------

\begin{lemma}
  \label{lem:fo_eho}
  Let $(S,m)$ be a multiset in a commutative ring $R$ such that all
  nonzero elements $s-s^*$ in $S-S$ are units, and let
  $t=d(S)-1$. Then there exist elements $\alpha(s,u)\in R$ for all
  $s\in S$ and $0\le u<m(s)$ with the following property: for every
  $\ell\ge0$,
  $$
  \sum_{s\in S} \sum_{0\le u<m(s)} \alpha(s,u) \binom{\ell}{u} s^{\ell-u}
  = \begin{cases}
    0 & \text{if $\ell<t$;} \\
    1 & \text{if $\ell=t$;} \\
    * & \text{if $\ell>t$.} \\
  \end{cases}
  $$
  (Here the symbol $*$ means ``undetermined''.)
\end{lemma}

\begin{proof}
  Take the polynomials $h^{(s,u)}(x)$ provided by
  Lemma~\ref{lem:h_base_pol}, and let $\alpha(s,u)$ be the coefficient of
  $x^t$ in the polynomial $h^{(s,u)}(x)$.

  For every $s\in S$ we have
  $$
  x^\ell = \big(s+(x-s)\big)^\ell
  = \sum_{u=0}^\infty \binom{\ell}{u}s^{\ell-u} (x-s)^u,
  $$
  and therefore
  $$
  \big(x^\ell\big)_u(s) = \binom{\ell}{u}s^{\ell-u}
  $$
  for every $u\ge0$. (For $u>\ell$ we have $\binom{\ell}{u}=0$ and the
  negative exponent in $s^{\ell-u}$ does not matter.)
  
  Now take an arbitrary $\ell<d(S)$, and apply Lemma~\ref{lem:h_int}
  to the values $y_{s,u}=\binom{\ell}{u}s^{\ell-u}$. From these
  values, Lemma~\ref{lem:h_int} reconstructs the polynomial $x^\ell$:
  $$
  \sum_{s\in S} \sum_{u<m(s)} \binom{\ell}{u}s^{\ell-u} h^{(s,u)}(x) = x^\ell.
  $$

  Comparing the coefficients of $x^t$, we get
  $$
  \sum_{s\in S} \sum_{u<m(s)} \alpha(s,u) \binom{\ell}{u}s^{\ell-u}
  = \begin{cases}
    0 & \text{if $\ell<t$;} \\
    1 & \text{if $\ell=t$.} \\
  \end{cases}
  $$
\end{proof}

%-------------------------------------------------------------------

\begin{proof}[Proof of Theorem \ref{modulo_multi}]
  
  Without loss of generality, we may assume that $t_i=d(S_i)-1$ for
  every $i=1,2,\dots,n$.

  Expand $f$ as $f(\ve x)=\sum\limits_{\ve k} c_{\ve k} \ve x^{\ve k}$. For
  every $\ve s\in S$, from
  \begin{gather*}
    f(\ve x)
    =
    \sum_{\ve k} c_{\ve k} \big(\ve s+(\ve x-\ve s)\big)^{\ve k}
    =
    \sum_{\ve k} c_{\ve k} \sum_{\ve u} 
    \left(\prod_{i=1}^n \binom{k_i}{u_i}s_i^{k_i-u_i}\right)
    (\ve x-\ve s)^{\ve u}
    = \\ =
    \sum_{\ve u}
    \left(\sum_{\ve k} c_{\ve k} \prod_{i=1}^n
      \binom{k_i}{u_i}s_i^{k_i-u_i}\right) (\ve x-\ve s)^{\ve u}
  \end{gather*}
  we get
  $$
  f_{\ve u}(\ve s)
  = \sum_{\ve k} c_{\ve k} \prod_{i=1}^n \binom{k_i}{u_i}s_i^{k_i-u_i}.
  $$

  For each $i$, by Lemma~\ref{lem:fo_eho}, there exist elements
  $\alpha_i(s,u)\in R$ for all $s\in S_i$ and $0\le u<m_i(s)$ such
  that
  $$
  \sum_{s\in S_i} \sum_{0\le u<m_i(s)} \alpha_i(s,u) \binom{\ell}{u} s^{\ell-u}
  = \begin{cases}
    0 & \text{if $\ell<t_i$;} \\
    1 & \text{if $\ell=t_i$;} \\
    * & \text{if $\ell>t_i$.} \\
  \end{cases}
  $$
  Now let $\alpha (\ve s,\ve
  u)=\prod\limits_{i=1}^n\alpha_i(s_i,u_i)$ and consider the following
  expression:
  \begin{gather*}
    \sum_{\ve s\in S} \sum_{\ve u<m(\ve s)}
    \alpha(\ve s,\ve u) f_{\ve u}(\ve s)
    =
    \sum_{\ve s\in S} \sum_{\ve u\le m(\ve s)}
    \left(\prod_{i=1}^n\alpha_i(s_i,u_i)\right)
    \sum_{\ve k} c_{\ve k} \prod_{i=1}^n \binom{k_i}{u_i}s_i^{k_i-u_i}
    =
    \\
    =
    \sum_{\ve k} c_{\ve k} 
    \sum_{s_1\in S_1} \sum_{u_1<m_1(s_1)} \ldots
    \sum_{s_n\in S_n} \sum_{u_n<m_n(s_n)}
    \prod_{i=1}^n \left( \alpha_i(s_i,u_i)
      \binom{k_i}{u_i}s_i^{k_i-u_i} \right) 
    =
    \\
    =
    \sum_{\ve k} c_{\ve k} \prod_{i=1}^n \left( 
      \sum_{s_i\in S_i} \sum_{u_i<m_i(s_i)}
      \alpha_i(s_i,u_i) \binom{k_i}{u_i}s_i^{k_i-u_i} \right) 
    =
    \sum_{\ve k} c_{\ve k} \prod_{i=1}^n \left\{
      \begin{matrix}
        0 & \text{if $k_i<t_i$} \\
        1 & \text{if $k_i=t_i$} \\
        * & \text{if $k_i>t_i$} \\
      \end{matrix} \right\}.
  \end{gather*}
  Since $\deg f=t_1+\ldots+t_n$, the last product is zero except for
  $\ve k=\ve t$ (when every factor is $1$). Therefore, we have
  $$
  \sum_{\ve s\in S} \sum_{\ve u<m(\ve s)}
  \alpha(\ve s,\ve u) f_{\ve u}(\ve s)
  =
  c_{\ve t}.
  $$
  On the left-hand side, there stands a linear combination of the
  values $f_{\ve u}(\ve s)$. Since $c_{\ve t}\ne0$ on the right-hand
  side, there must be at least one nonzero among these values.
\end{proof}

%-------------------------------------------------------------------

\section{An alternative proof for Theorem~\ref{theorem 1}}

Here we intend to give a proof of
Theorem~\ref{theorem 1} which follows closely the original line of reasoning 
from Alon \cite{Alon1}. 

\bigskip

\begin{proof}[Proof of Theorem~\ref{theorem 2}]

We denote by $V$ the $\B$ module of all
functions from $S$ to $\B$.  $V$ is a free 
$\B$ module,
\[\mbox{rank}_{\B}V=|S|=\prod_{i=1}^n d_i,\]
where $|S_i|=\deg g_i=d_i$. In fact, for  $\ve s\in S$ we denote by $f_{(\ve
s)}$ the
$S\longrightarrow \B$ function taking value $1$ at  $\ve s$ and $0$
everywhere else in $S$. Then the set 
$F=\{f_{(\ve s)} | ~ \ve s\in S\}$ is a free generating set of $V$ over $\B$, 
and 
$|F|=\prod_{i=1}^n d_i=|S|$. 
Next we observe, that every $f_{(\ve s)}$ can be
written as a polynomial from $\Bx$, using interpolation. For 
$\ve s=(s_1,\ldots, s_n)\in S$ we have 

\[ f_{(\ve s)}(\ve x)=\prod_{i=1}^n \left( \prod_{\alpha \in S_i, \alpha \neq
s_i}(x_i-\alpha)(s_i-\alpha)^{-1}\right) . \]

Note that since $s_i\not=\alpha \in S_i$, the element
$s_i-\alpha$ is a unit in $\B$, hence 
the definition of $f_{(\ve s)}(\ve x)$ makes sense. 
Consider the following set of monomials

\[M=\{\ve x^{\ve w};~ w_i\leq d_i-1, ~i=1,\dots,n\}.\]

Set $\mathcal{G}=\{g_1(x_1),\dots,g_n(x_n)\}$.
An arbitrary polynomial $f$ from $\Bx$ can be reduced 
with $\mathcal{G}$. This means that an occurrence of the monomial 
$x_i^{d_i}$ is replaced by $-(g_i(x_i)-x_i^{d_i})$ as long as it is possible. 
Note that this reduction does not change $f$ as a function on $S$. Clearly 
any $f$ can be reduced into a $\B$-linear combination of monomials from 
$M$. In particular, the elements of $F$ are reduced this way into a
collection of $|S|$ polynomials which are independent over $\B$. Using 
also that $|M|=|S|$, we infer that $M$, as a set of functions from $S$ to
$\B$, is also linearly independent over $\B$. 

Now consider an arbitrary polynomial $f$ from $\Bx$, and reduce
$f$  with $\mathcal{G}$ as much as possible.
Denote the resulting (reduced) polynomial by $r$, which is a $\B$-linear 
combination of   
monomials from $M$. The fact that $f$ reduces to $r$ means that there are 
polynomials $h_1,\dots,h_n \in
\Bx$ such that $\deg (h_i)\leq \deg f - d_i$ for all $i$, the degree
of $r$ is less than $d_i$ in every $x_i$, and
\[f\x =r\x + \sum_{i=1}^{n} h_i\x g_i(x_i).\]
Because of the linear independence of the monomial functions  from $M$, we
have that $f\in I(S)$ if and only if $r$ is the all zero linear
combination. This concludes the proof. 
\end{proof}

\medskip 

We remark that the proof actually gives that ${\cal G}$ is a Gr\"obner basis
of $I(S)$ with respect to an arbitrary term order on $\Bx$. For 
an introduction to Gr\"obner bases the reader is referred to \cite{AL}.

\medskip

From Theorem \ref{theorem 2} the original argument of Alon gives 
Theorem \ref{theorem 1} quite simply. Below we reproduce Alon's proof for the 
convenience of the reader.

\begin{proof}[An alternative proof of Theorem \ref{theorem 1}] 
Clearly we may assume that $|S_i|=t_i+1$ for all
$i$. Suppose that the result is false, i.e. $f\in I(S)$, and define  
$g_i(x_i)=\prod_{s\in S_i} (x_i-s)$. By Theorem \ref{theorem 2} 
there are polynomials $h_1,\dots,h_n \in \Bx$ such that
$\deg (h_j)\leq \sum_{i=1}^n t_i - \deg (g_j)$ for all $j$, for which

\[f=\sum_{i=1}^n h_i g_i.\]

Here the degree of $h_ig_i$ is at most $deg(f)$, and if there are
any monomials of degree $deg(f)$ in it, then they are divisible by
$x_i^{t_i+1}$. It follows that the coefficient of $\prod_{i=1}^n  
x_i^{t_i}$ on the right hand side is zero. However by our
assumption the coefficient of $\prod_{i=1}^n x_i^{t_i}$ on the
left hand side is nonzero, and this contradiction completes the  
proof. 
\end{proof}

\section{Two applications}

We can extend a result of Alon and F\"uredi \cite{AF} on the  covering of a
discrete cube by hyperplanes in the following way. (The original result 
is the special case when every multiplicity is 1.)

\begin{theorem}
  Let $(S_1, m_1), \ldots ,(S_n,m_n)$ be finite multisets from the
  field $\F$. Suppose that $0\in S_i$, with $m_i(0)=1$ for every $i$,
  and $H_1,\dots,H_k$ are hyperplanes in
  $\F^n$ such that every point $\ve s\in S\setminus\{\ve0\}$ is covered
by at least $|m(\ve s)|-n+1$ hyperplanes and 
the point $\ve0$ is not covered by any of the hyperplanes.
Then $k\geq d(S_1)+d(S_2)+\cdots + d(S_n)-n$.
\end{theorem}

\begin{proof} Let  $\ell_j(\ve x)$ be a
linear polynomial defining the hyperplane
  $H_j$, set  $f(\ve x)=\prod\limits_{j=1}^k\ell_j(\ve x)$, and
$t_i=d(S_i)-1$.

  Let
  $$ 
  P(\ve x) = \prod_{i=1}^n\prod_{s\in S_i\setminus \{0\}}(x_i-s)^{m_i(s)}
  $$ 
  and
  $$ 
  F(\ve x) = P(\ve x) - \frac{P(\ve0)}{f(\ve0)}f(\ve x).
  $$
  Note that we have $f(\ve0)\ne0$, because the
  hyperplanes do not cover $\ve0$.
  If the statement is false, then the degree of $F$ is $t_1+t_2+\cdots
  + t_n$ and the coefficient of $x_1^{t_1}\cdots x_n^{t_n}$ is
  $1$. Theorem~\ref{nonvanish} applies  with $(S_1, m_1), \ldots
  ,(S_n,m_n)$ and $t_1,\ldots ,t_n$: there exists a vector $\ve s \in  
  S$, and an exponent vector $\ve u < m(\ve s )$ such that $F_{\ve u}(\ve s)
 \not=0$. 
 We observe that $\ve   
  s$ cannot be $\ve 0$, because $F(\ve 0)=0$. Thus $\ve s$ must have at
  least one nonzero coordinate, implying that
  $ P_{\ve u}(\ve s)=0$. 

Moreover, as $\ve s$ is a nonzero vector,  $f(\ve x)$ must vanish
  at $\ve s$ at least $|m(\ve s)|-n+1$ times, implying that 
  $f_{\ve u}(\ve s)=0$  (expand the product at $\ve s$; for
  every term $(\ve x -\ve s )^{\ve v} $ obtained there will be an
  index $j$ such that $v_j\geq m_j(s_j)$). These facts 
imply that $F_{\ve u}(\ve s)=0$, a contradiction. This finishes the proof.
\end{proof}

Next, as an application of Theorem \ref{theorem 1}, we present a
generalization of Theorem 6.3. from \cite{Alon1} to the Boolean cube
over a commutative ring $\B$. By a hyperplane $H$ in $\B^n$ we understand 
the set of zeros of a polynomial of the form $a_1x_1+\cdots +a_nx_n-b:=(\ve
a, \ve x) -b$, where $a_i,b\in \B$.

\begin{theorem}\label{theorem 3}
Let $\B$ be a commutative ring, and let  $H_1,\ldots ,H_m$ be 
hyperplanes in $\B^n$ such that 
$H_1,\dots,H_m$ cover all the vertices of the unit cube
$\{0,1\}^n\subseteq \B^n$, with the exception of $\ve 0$.  
Let $(\ve a^i,\ve x)-b_i$ be the polynomial defining $H_i$. 
If $\prod_{i=1}^m b_i\neq 0$, 
then $m\geq n$.
\end{theorem}

\begin{proof} The proof is essentially the same as the one in
\cite{Alon1}. Assume that the assertion is false: $m<n$, and
consider the polynomial

\[P\x=(-1)^{n+m+1} \prod_{j=1}^m b_j 
\prod_{i=1}^n(x_i-1)-\prod_{i=1}^m[(\ve
a ^i,\ve x)-b_i]. \]
The degree of this polynomial is clearly $n$, and the coefficient of
$\prod_{i=1}^nx_i$ in $P$ is $(-1)^{n+m+1}\prod_{j=1}^m b_j$, which is
nonzero by our assumption. By applying Theorem \ref{theorem 1}   
to $S_i=\{0,1\}$, $t_i=1$, we obtain a point $\ve s\in \{0,1\}^n$ for
which $P(\ve s) \neq 0$. This point is not the all zero vector, as $P$  
vanishes on $\ve 0$. 
But otherwise $(\ve a_i,\ve s)-b_i=0$ for some $i$ (as $\ve s$ 
is covered by an $H_i$), implying that $P$ does vanish on
$\ve s$, a contradiction. 
\end{proof}

Alon and F\"uredi have obtained the preceding 
statement for the case
$\B =\F$,  using the original nonvanishing argument.
If we put $\B=\mathbb Z_n$ for some square free integer $n\in
\mathbb{N}$, then an application of the original
nonvanishing theorem  for a suitable  prime factor of $n$ proves the 
statement. However, if $n$ has square factors, then Theorem \ref{theorem 3} 
appears to give a new result.

\section{Concluding remarks}

Our interest in developing a version of the nonvanishing theorem for
multisets has grown out of an attempt to prove Snevily's conjecture
\cite{Snevily} in a
particular case.

Let $G$ be a finite group of odd order and suppose that
$a_1,\dots,a_k\in G$ are pairwise distinct and $b_1,\dots,b_k\in G$
are pairwise distinct. Snevily's Conjecture states that there is a
permutation $\pi$ of the indices $1,2,\dots, k$ for which
$a_1b_{\pi(1)},a_2b_{\pi(2)},\dots,a_kb_{\pi(k)}$ are pairwise
distinct. The conjecture has been proved for cyclic groups of prime
order by Alon \cite{Alon3}, for cyclic groups by Dasgupta, K\'arolyi, Serra,
and Szegedy \cite{DKSSz}
and recently for all commutative groups by Arsovski \cite{Arsovski}.
Note that Alon's proof, which is based on the Combinatorial
Nullstellensatz, allows one of the sequences $a_1,\dots,a_k$ and
$b_1,\dots,b_k$ to contain repeated elements when $k<|G|$.

\medskip

Our approach was the following. Let $N$ be a normal subgroup of $G$.
We look for the permutation to be applied to the factor group $G/N$
first. Of course, in the factor group, some of the cosets
$a_1N,\dots,a_nN$ and $b_1N,\dots,b_nN$ may coincide to each other; it
is even possible that $a_1N,\dots,a_nN$ all coincide and   
$b_1N,\dots,b_nN$ all coincide. So we must allow
$(a_iN)(b_{\pi(i)}N)=(a_jN)(b_{\pi(j)}N)$ in some cases.

The idea is to allow $(a_iN)(b_{\pi(i)}N)=(a_jN)(b_{\pi(j)}N)$ only in
those cases, when $a_iN=a_jN$ and $b_{\pi(i)}N=b_{\pi(j)}N$ holds. Suppose
that we found such a permutation $\pi$. If some classes
$(a_{i_1}N)(b_{\pi(i_1)}N),(a_{i_2}N)(b_{\pi(i_2)}N),
\dots,(a_{i_\ell}N)(b_{\pi(i_\ell)}N)$
coincide, then the elements $a_{i_1},\dots,a_{i_\ell}$ are all in the same
coset $cN$, and similarly $b_{\pi(i_1)},\dots,b_{\pi(i_\ell)}$ are in
the same coset $Nd$. Then we could apply Snevily's Conjecture
inductively to the sequences $c^{-1}a_{i_1},\dots,c^{-1}a_{i_\ell}$  
and $b_{\pi(i_1)}d^{-1},\dots,b_{\pi(i_\ell)}d^{-1}$ which all lie in
$N$. By the induction hypothesis, the values
$\pi(j_1),\dots,\pi(j_\ell)$ can be permuted in such a way, that
$c^{-1}a_{i_1}b_{\pi(i_1)}d^{-1},\dots,
c^{-1}a_{i_\ell}b_{\pi(i_\ell)}d^{-1}$
are pairwise distinct. Hence the elements
$a_{i_1}b_{\pi(i_1)},\dots,a_{i_\ell}b_{\pi(i_\ell)}$ are all in the
coset $(cd)N=cNd$, and they will be pairwise distinct.

If we choose $N$ to be a maximal normal subgroup of $G$, then the order
$|G/N|=p$ will be a prime ($G$ is solvable, since $|G|$ is odd). In  
this case we can use the additive group of the prime field 
$\Fp$ and re-formulate the 
existence of the desired permutation.

\begin{conjecture}
  \label{conj}
  Suppose that $k\le p$ and $a_1,\dots,a_k\in\Fp$ and
  $b_1,\dots,b_k\in\Fp$ are arbitrary elements. Then there is a
  permutation $\pi$ of the indices $1,2,\dots, k$ with the following
  property: $a_i=a_j$ and $b_{\pi(i)}=b_{\pi(j)}$ holds whenever   
  $a_i+b_{\pi(i)}=a_j+b_{\pi(j)}$.
\end{conjecture}

This conjecture would imply Snevily's conjecture for all cases when
$k$ is not greater than the smallest prime divisor of $|G|$.

\medskip

To use the polynomial method to prove Conjecture~\ref{conj} and
similar statements, it appears that one is required to  
handle multiple values in the
sequences $a_1,\dots,a_k$ and $b_1,\dots,b_k$.

\bigskip
The requirements in Theorems \ref{theorem 1} and 
\ref{modulo_multi} about the invertibility in $\B$ of the differences 
$s-s^*$ can be somewhat relaxed. The theorems still hold if we assume, 
instead of
invertibility, that the multiplicative monoid $D\subset \B$ generated by 
the differences $s-s^*$ and $f_{(t_1,\ldots ,t_n)}$ does not contain 0. In
this case the map from $\B$ to its ring of fractions  $D^{-1}\B$ lifts our 
proofs from  $D^{-1}\B$ to $\B$ (the reader is referred to \cite{AM} for
basic facts on rings of fractions). Actually, it suffices to assume that 
a certain specific product from $D$ is not zero.

%-------------------------------------------------------------------


\begin{thebibliography}{MM}



\bibitem{AL} W. W. Adams, P. Loustaunau, {\em An introduction to Gr\"obner
bases,} American Mathematical Society, 1994.

\bibitem{Alon1} N. Alon, 
Combinatorial Nullstellensatz, {\em Combinatorics, Probability and
Computing} {\bf 8} (1999), 7-29.

\bibitem{Alon3}
N. Alon, Additive Latin transversals, {\em   Israel J. Math.} {\bf 117}  
(2000), 125--130. 

\bibitem{AF} N. Alon, Z. F\"uredi, Covering the cube by affine hyperplanes, 
{\em European J. Combinatorics} {\bf 14} (1993), 79--83.

\bibitem{Arsovski} B. Arsovski, A proof of Snevily's conjecture,
{\em Israel Journal of Math. }  {\bf 182}  (2011), 505--508. 


\bibitem{AM} M. F. Atiyah, I. G. Macdonald, {\em Introduction to commutative
algebra,} Addison-Wesley, 1969.

\bibitem{CLM}
M. C\'amara, A. Llad\'o, J.Moragas, On a conjecture of Graham and 
H\" aggkvist with the polynomial method, {\em European Journal of
 Combinatorics} {\bf 30} (2009), 1585--1592.

\bibitem{DKSSz}
S. Dasgupta, Gy. K\'arolyi, O. Serra, 
B. Szegedy, 
Transversals of additive Latin squares,
{\em Israel J. Math.} {\bf  126} (2001), 17--28.



\bibitem{Felszeghy} B. Felszeghy, On the solvability of some special 
equations over finite fields, {\em Publicationes  Mathematicae Debrecen} 
{\bf 68} (2006), 
15--23. 

\bibitem{GT}
B. Green, T. Tao, The distribution of polynomials over finite fields, with
applications to the Gowers norms, 
{\em Contributions to  Discrete Mathematics} {\bf 4} (2009), 1--36.


\bibitem{KarolyiI} Gy. K\'arolyi,
An inverse theorem for the restricted set addition in abelian
groups, {\em  Journal of Algebra}  {\bf 290}  (2005), 557--593. 

 \bibitem{KarolyiR} Gy.  K\'arolyi, Restricted set addition: the
exceptional case of the Erd\H {o}s-Heilbronn conjecture, {\em Journal of
Combinatorial Theory, Ser. A} {\bf 116} (2009), 741--746.

\bibitem{KR} G. K\'os, L. R\'onyai, Alon's Nullstellensatz for multisets,
to appear in {\em Combinatorica}.
 
\bibitem{Michalek} M. Micha{\l}ek, A short proof of Combinatorial
Nullstellensatz, {\em Amer. Math. Monthly} {\bf 117} (2010), 821--823.



\bibitem{PS}
H. Pan, Z-W. Sun,  A new extension of the Erd\H {o}s-Heilbronn
conjecture, {\em Journal of Combinatorial Theory, Ser. A} {\bf 116} (2009),
1374--1381.

\bibitem{Snevily} H. Snevily, Unsolved Problems: The Cayley Addition Table of 
$\mathbb Z_n$, {\em Amer. Math. Monthly} {\bf  106}  (1999), 
584-585.

\bibitem{Sun} Z-W. Sun, On value sets of polynomials over a field, {\em
Finite Fields and Applications} {\bf 14} (2008), 470--481.

\end{thebibliography}
\end{document}